\documentclass[10 pt, conference]{ieeeconf}
\usepackage[utf8]{inputenc}

\usepackage[bookmarks=true]{hyperref}
\usepackage{xcolor}
\usepackage{graphicx}
\usepackage{dsfont} 
\usepackage{amsmath}
\usepackage{amssymb}
\usepackage{url}

\usepackage{color}
\usepackage{subcaption}
\usepackage[utf8]{inputenc}
\usepackage[T1]{fontenc}
\usepackage[T1]{fontenc}
\IEEEoverridecommandlockouts

\usepackage{cleveref}

\newtheorem{assumption}{Assumption}
\newtheorem{definition}{Definition}
\newtheorem{lemma}{Lemma}
\newtheorem{theorem}{Theorem}

\usepackage{comment}
\usepackage{mathrsfs}
\usepackage[textfont=md,font=footnotesize]{caption}

\DeclareMathOperator{\rank}{rank}

\DeclareMathOperator{\expval}{\mathbb{E}}

\newcommand{\doi}[1]{\href{http://dx.doi.org/#1}{\normalsize{\textsc{doi:}}~\nolinkurl{#1}}}
\newcommand{\arxiv}[1]{\href{http://arxiv.org/abs/#1}{\normalsize{\textsc{arxiv:}}~\nolinkurl{#1}}}

\renewcommand{\epsilon}{\varepsilon}
\renewcommand{\phi}{\varphi}

\newcommand{\R}{\mathbb{R}}

\newcommand{\N}{\mathbb{N}}

\let\originalleft\left
\let\originalright\right
\renewcommand{\left}{\mathopen{}\mathclose\bgroup\originalleft}
\renewcommand{\right}{\aftergroup\egroup\originalright}

\def\clap#1{\hbox to 0pt{\hss#1\hss}}

\newcommand{\norm}[1]{\left\lVert #1\right\rVert}

\newcommand{\set}[1]{\left\{ #1\right\}}

\DeclareMathAlphabet{\mathpzc}{OT1}{pzc}{m}{it}

\newcommand{\dfknote}[1]%
    {\textcolor{orange}{\textbf{DFK: #1}}}
\newcommand{\twnote}[1]%
    {\textcolor{cyan}{\textbf{TW: #1}}}
\newcommand{\emnote}[1]%
    {\textcolor{blue}{\textbf{EM: #1}}}
    
\title{Feedback Linearization for Uncertain Systems via Reinforcement Learning}

\author{
Tyler Westenbroek*, David Fridovich-Keil*, Eric Mazumdar*, Shreyas Arora, Valmik Prabhu,\\ S. Shankar Sastry, and Claire J. Tomlin
\thanks{
This work was supported by HICON-LEARN (design of HIgh CONfidence LEARNing-enabled systems), Defense Advanced Research Projects Agency award number FA8750-18-C-0101, and Provable High Confidence Human Robot Interactions, Office of Naval Research award number N00014-19-1-2066.
EECS, UC Berkeley. \href{mailto:westenbroekt@eecs.berkeley.edu}{\tt \small{westenbroekt@berkeley.edu}}.\newline
$^*$ indicates equal contribution.}
}

\begin{document}
\maketitle

\begin{abstract}
We present a novel approach to control design for nonlinear systems which leverages model-free policy optimization techniques to learn a linearizing controller for a physical plant with unknown dynamics. Feedback linearization is a technique from nonlinear control which renders the input-output dynamics of a nonlinear plant \emph{linear} under application of an appropriate feedback controller. Once a linearizing controller has been constructed, desired output trajectories for the nonlinear plant can be tracked using a variety of linear control techniques. However, the calculation of a linearizing controller requires a precise dynamics model for the system. As a result, model-based approaches for learning exact linearizing controllers generally require a simple, highly structured model of the system with easily identifiable parameters. In contrast, the model-free approach presented in this paper is able to approximate the linearizing controller for the plant using general function approximation architectures. Specifically, we formulate a continuous-time optimization problem over the parameters of a learned linearizing controller whose optima are the set of parameters which best linearize the plant. We derive conditions under which the learning problem is (strongly) convex and provide guarantees which ensure the true linearizing controller for the plant is recovered. We then discuss how model-free policy optimization algorithms can be used to solve a discrete-time approximation to the problem using data collected from the real-world plant. The utility of the framework is demonstrated in simulation and on a real-world robotic platform. 
\end{abstract}
\section{Introduction}
Geometric nonlinear control theory has developed a powerful set of feedback architectures which exploit the underlying structure of a control system to simplify downstream tasks such as trajectory generation and tracking \cite{sastry2013nonlinear,isidori2013nonlinear}. However, geometric controllers often require an accurate model for the system or a simple parameterization of the dynamics  which can be readily identified. Meanwhile, the model-free reinforcement learning literature \cite{bertsekas1996neuro, sutton2018reinforcement, sutton2000policy} has sought to automatically compute optimal feedback controllers for unknown systems without relying on structural  assumptions about the dynamics. However, despite a recent resurgence of research into these methods \cite{schulman2015trust, lillicrap2015continuous, schulman2017proximal}, their poor sample complexity has thus-far limited their applicability to many robotics applications. In this paper we unify these disparate approaches by using model-free reinforcement learning algorithms to optimize the performance of a class of geometric controllers for a plant with unknown dynamics, observing improvement in the performance of the trained controllers with practical amounts of data collects from the plant. 

Specifically, this paper focuses on a geometric technique called feedback linearization, which renders the input-output dynamics of a nonlinear system \emph{linear} under the application of an appropriately chosen feedback controller. Once a linearizing controller has been constructed, it can be used in conjunction with efficient tools from linear systems theory to generate and track desired output trajectories for the full nonlinear plant \cite{ martin2003flat, kalman1960contributions, borrelli2017predictive}. Despite its widespread use throughout robotics \cite{grizzle2001asymptotically, ames2014rapidly, mellinger2011minimum, sastry2013nonlinear, isidori2013nonlinear,} the primary drawback of feedback linearization is that exact cancellation of the system's nonlinearities requires a precise model of the dynamics. Complex phenomena such as friction, aerodynamic drag, and internal actuator dynamics yield nonlinearities which may be challenging to model or identify. While there have been extensive efforts to construct exact linearizing controllers for unknown plants using extensions of adaptive linear control theory \cite{sastry2011adaptive, craig1987adaptive, sastry1989adaptive, nam1988model, kanellakopoulos1991systematic, umlauft2017feedback, chowdhary2014bayesian, chowdhary2013bayesian}, these methods require a highly structured representation of the system's nonlinearites and employ complicated parameter update schemes to avoid singularities in the learned controller. A more detailed discussion of these methods is provided in Section \ref{sec:literature}.

\begin{figure}[t!]
    \centering
    \includegraphics[trim=6cm 7cm 6cm 7cm, clip=true, width=\columnwidth]{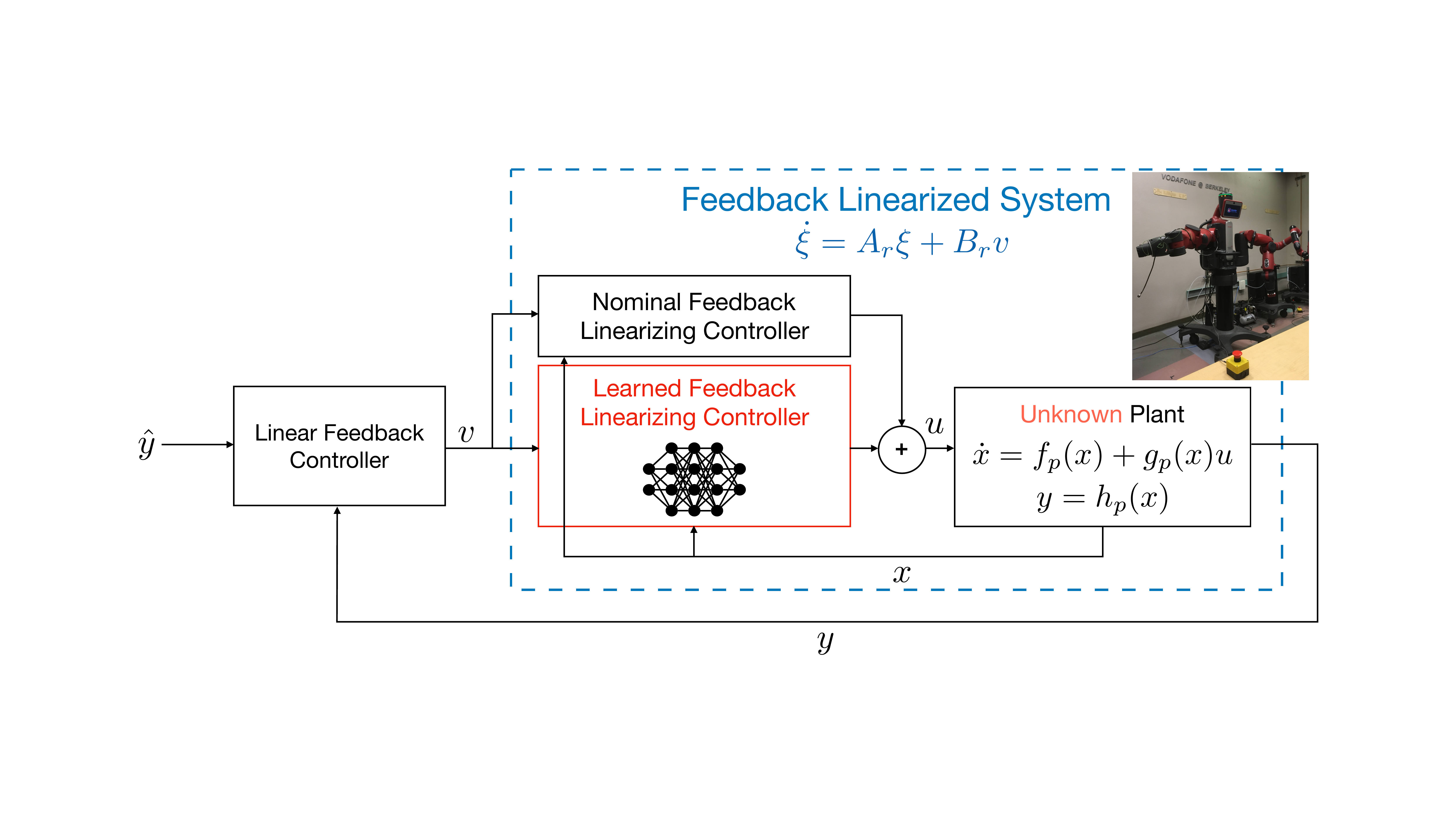}
    \caption{Schematic diagram of our framework. By learning an appropriate feedback linearizing controller, we render an initially unknown nonlinear system (with state $x$, output $y$, and input $u$) \emph{linear} in an auxiliary input $v$. The framework can make use of a nominal dynamics model (if available).}
    \label{fig:front}
    \vspace{-0.75cm}
\end{figure}

In sharp contrast to the standard methods discussed above, we propose a general framework for learning a linearizing controller for a plant with unknown dynamics using model-free policy optimization techniques. Our approach requires only that the order of the relationship between the inputs and outputs is known, can utilize arbitrary function approximation schemes, and remains singularity free during training. Concretely, our approach begins by constructing a linearizing controller for an approximate dynamics model of the plant (if available), and then augments this fixed model-based controller with a learned affine correction term. We then define a continuous-time optimization problem over the learned parameters which selects parameters which better linearize the unknown plant. We provide conditions on the structure of the learned component which  ensure that the learning problem is (strongly) convex. Our analysis draws connections between the familiar persistency-of-excitation conditions from the adaptive control literature and our formulation. Finally, we discuss how off-the-shelf reinforcement learning algorithms can be used to solve discrete-time approximations to the optimization problem, and demonstrate the utility of our framework by learning linearizing controllers for highly uncertain robotic systems in both simulation and on real-world hardware. 

The rest of the paper is structured as follows. Section \ref{sec:FBL} introduces feedback linearization and discusses prior approaches for learning linearizing controllers. Section \ref{sec:learning} details our approach, provides our theoretical results, and discusses practical algorithms for solving the learning problem. Our simulated and real-world robotic examples are presented in Section \ref{sec:examples}, and Section \ref{sec:discussion} provides closing remarks and discusses   future work.

\section{Feedback Linearization} \label{sec:FBL}

This section outlines how to compute an input-output linearizing controller for a known dynamics model and discusses prior data-driven methods for learning a linearizing controller for an unknown plant.  Due to space constraints, we refer the interested reader to \cite[Chapter 9]{sastry2013nonlinear} for a more complete introduction to feedback linearization. 
 In this paper, we consider square control-affine systems of the form
\begin{equation}
\label{eq:mimo}
\begin{aligned}
    \dot{x} &= f(x) + g(x)u \\
    y &= h(x),
\end{aligned}
\end{equation}
with sate $x \in \R^n$, input $u \in \R^q$  and output $y \in \R^q$. The mappings $f\colon \R^n \to \R^n$, $g \colon \R^n \to \R^{n \times q}$ and $h \colon \R^n \to \R^q$ are each assumed to be smooth. We restrict our attention to a compact subset $D \subset \R^n$ of the state-space.
\subsection{Single-input single-output systems}
We begin by introducing feedback linearization for single-input, single-output (SISO) systems (i.e.,  $q =1 $). We begin by examining the first time derivative of the output:
\begin{align*}\label{eq:nominal}
\dot{y} &= \frac{dh}{dx}(x) \cdot \Big(f(x) + g(x) u\Big)   \\
&= \underbrace{\frac{dh}{dx}(x) \cdot f(x)}_{L_{f}h(x)} + \underbrace{\frac{dh}{dx}(x) \cdot g(x)}_{L_{g} h(x)}u
\end{align*}
Here the terms $L_{f}h(x)$ and $L_{g}h(x)$ are known as \emph{Lie derivatives} \cite{sastry2013nonlinear}, and capture the rate of change of $y= h(x)$ along the vector fields $f$ and $g$, respectively. In the case that $L_{g}h(x) \neq0$ for each $x \in D$, we can apply the control law
\begin{equation}\label{eq:fb1}
    u(x,v) = \frac{1}{L_{g}h(x)} (-L_{f}h(x) + v)~,
\end{equation}
which exactly `cancels out' the nonlinearities of the system and enforces the linear relationship $\dot{y} = v$. After application of this linearizing control law, we can now control the output trajectory $y(t)$ through its first derivative. However if the input does not affect the first time derivative of the output {(that is, if $L_gh \equiv 0$)} then the control law \eqref{eq:fb1} will be undefined. In general, we can differentiate $y$ multiple times until the input appears. Assuming that the input does not appear the first $\gamma-1$ times we differentiate the output, the $\gamma$-th time derivative of $y$ will be of the form
\begin{equation*}
    y^{(\gamma)} = L_f^{\gamma}h(x) + L_gL_f^{\gamma-1}h(x)u.
\end{equation*}
Here, $L_f^{\gamma}h(x)$ and $L_gL_f^{\gamma-1}h(x)$ are higher order Lie derivatives (see \cite[Chapter 9]{sastry2013nonlinear} for more details). If $L_gL_f^{\gamma-1}h(x) \neq0$ for each $x \in D$ then the control law 
\begin{equation*}
    u(x,v) = \frac{1}{L_gL_f^{\gamma -1}h(x)}\big(-L_f^{\gamma}h(x) + v\big)
\end{equation*}
enforces the trivial linear relationship $y^{(\gamma)} = v$. We refer to $\gamma$ as the \emph{relative degree} of the nonlinear system, which is simply the order of its input-output relationship. 

\subsection{Multiple-input multiple-output systems}
Next, we consider (square) multiple-input, multiple-output (MIMO) systems where $q > 1$.  As in the SISO case,  we differentiate each of the output channels until at least one input appears.  Let $\gamma_j$ be the number of times we need to differentiate $y_j$ (the $j$-th entry of $y$) for at least one input to appear. Combining the resulting expressions for each of the outputs yields an input-output relationship of the form
\begin{equation}
\label{eq:first_A_b}
   [y_1^{(\gamma_1)}, \dots, y_q^{(\gamma_q)}]^T = b(x) +  A(x) u.
\end{equation}
Here, the matrix $A(x) \in \R^{q \times q}$ is known as the \emph{decoupling matrix} and the vector $b(x) \in \mathbb{R}^q$ is known as the \emph{drift term}. If $A(x)$ is bounded away from singularity for each $x \in D$ then we observe that the control law
\begin{equation}\label{eq:mimo_controller}
    u(x,v) = A^{-1}(x)(-b(x) + v)
\end{equation}
where $v \in \R^q$ yields the decoupled linear system
\begin{equation}\label{eq:decoupled_sys}
    [
    y_1^{(\gamma_1)}, y_2^{(\gamma_2)}, \dots, y_q^{(\gamma_q)}]^T = [v_1, v_2, \dots, v_q]^T,
\end{equation}
where $v_j$ is the $j$-th entry of $v$ and $y_j^{(\gamma_j)}$ is the $\gamma_j$-th derivative of the $j$-th output. We refer to $\gamma =(\gamma_1, \gamma_2, \dots, \gamma_q)$ as the \emph{vector relative degree} of the system. The decoupled dynamics \eqref{eq:decoupled_sys} are LTI and can be compactly represented with the \emph{reference model}
\begin{equation}\label{eq:reference}
    \dot{\xi}_r = A\xi_r + Bv_r,
\end{equation}
where $\xi_r = (y_1, \dot{y}_1, \dots, \dots , y_1^{\gamma_1-1}, \dots, y_q, \dots, y_q^{\gamma_q-1})$ and $A \in \R^{\|\gamma\| \times\ |\gamma\|}$ and $B \in \R^{|\gamma| \times q}$ are dynamics matrices with appropriate entries. The reference model can then be used to construct a desired trajectory $\xi_r(\cdot)$ for the output of the nonlinear system (and its derivatives), and to construct linear feedback controllers which can be used in conjunction with \eqref{eq:mimo_controller} to track the reference (see \cite[Theorem 9.14]{sastry2013nonlinear} for details). 

\subsection{Constructing linearizing controllers with data}\label{sec:literature}
The fundamental challenge model-based methods face when constructing a linearizing controller for an unknown plant is that it is very difficult to identify an estimate for the system's decoupling matrix (or possibly its inverse) which is guaranteed to be invertible. The predominant approaches for learning linearizing controllers are founded on the linear model reference adaptive control (MRAC) literature \cite{sastry2011adaptive}, and seek to recursively improve an estimate for the true linearizing controller using data collected from the plant \cite{sastry2011adaptive, craig1987adaptive, sastry1989adaptive, nam1988model, kanellakopoulos1991systematic, spooner1996stable, chen1995adaptive, yesildirek1994feedback}. These methods update the parameters for estimates of the decoupling matrix (or its inverse) online, but require that the estimated matrix remain invertible at all times. Typically, projection-based update rules are used to keep the parameters of the estimated matrix in a region which is singularity-free. However, the construction of these bounds requires a highly accurate yet simple parameterization of the system, and assumes that the true parameters of the system lie within some nominal set. 

One alternative approach \cite{kosmatopoulos1999switching, kosmatopoulos2002robust, bechlioulis2008robust} is to perturb the estimated linearizing control law to avoid singularities. These methods enable the use of more general function approximation schemes but sacrifice some tracking performance. Recently, non-parametric function approximators have been been used to learn a linearizing controller \cite{umlauft2017feedback, umlauft2019feedback}, but these methods still require structural assumptions to avoid singularities. In the following section we use model-free policy optimization algorithms to update the paramters of a learned linearizing controller while avoiding singularities.

\section{Directly Learning a Linearizing Controller}\label{sec:learning}
Our goal is to learn a linearizing controller for the plant
\begin{align}
\label{eq:plant}
    \dot{x}_p &= f_p(x_p) + g_p(x_p)u_p \\ \nonumber
    y_p &= h_p(x_p) 
\end{align} 
which is unknown. Our approach can incorporate nominal model
\begin{align}\label{eq:model}
    \dot{x}_m &= f_m(x_m) + g_m(x_m)u_m \\ \nonumber
    y_m & = h_m(x_m)
\end{align}
which represents our "best guess" for the true dynamics of the plant. We will assume that the plant and model have well-defined relative degrees $(\gamma_1^p, \dots, \gamma_q^p)$ and $(\gamma_1^m, \dots, \gamma_q^m)$, respectively, on some chosen set $D \subset \mathbb{R}^n$. We make the following assumption about our plant and model:
\begin{assumption}\label{ass:relative_degree1}\vspace{0.3cm}
The model system \eqref{eq:model} and plant \eqref{eq:plant} have the same relative degree on $D$ in the sense that $(\gamma_1^p, \dots, \gamma_q^p) = (\gamma_1^m, \dots, \gamma_q^m)$.
\end{assumption}\vspace{0.3cm}

This is a rather mild assumption, as the order of the relationship between the inputs and outputs of the plant can usually be inferred from first principles, even without a perfect model for the dynamics of the system. For example, for Lagrangian systems, such as the manipulator arms we consider in Section \ref{sec:examples}, we know the torques applied by an actuator produce an acceleration in the joints, even if we don't know the exact relationship between the two quantities. Moreover, with this assumption in place, we know there are linearizing controllers of the form
\begin{align*}
      u_m(x,v) &= \beta_m(x) + \alpha_m(x)v\\
       u_p(x,v) &= \beta_p(x) + \alpha_p(x)v
\end{align*}
 for the model and plant, respectively, which match the input-output dynamics of both systems to a common reference model \eqref{eq:reference}. Here, $\beta_p(x), \beta_m(x) \in \mathbb{R}^q$ and $\alpha_p(x), \alpha_m(x) \in \mathbb{R}^{q \times q}$. While we do not know $u_p$ \emph{a priori}, we do know that 
 \begin{align*}
     \beta_p(x) &= \beta_m(x) + \Delta \beta(x)\\
     \alpha_p(x) &= \alpha_m(x) + \Delta \alpha(x)
 \end{align*}
for some continuous functions $\Delta \beta$ and $\Delta \alpha$. We construct the following parameterized estimates for these functions:
\begin{equation*}
 \Delta \beta(x) \approx \beta_{\theta_1}(x) \ \ \ \ \ \ ~\Delta \alpha(x) \approx \alpha_{\theta_2}(x) 
\end{equation*}
Here, $\theta_1 \in \Theta_1 \subset \R^{K_1}
_1$ and $\theta_2 \in \Theta_2 \subset \R^{K_2}$ are parameters to be trained by running experiments on the plant. We will assume that $\Theta_1$ and $\Theta_2$ are convex sets, and we will frequently abbreviate $\theta = (\theta_1, \theta_2) \in \Theta_1 \times \Theta_2  : = \Theta$. We assume that $\beta_{\theta_1}$ and $\alpha_{\theta_2}$ are continuous in $x$ and continuously differentiable in $\theta_1$ and $\theta_2$, respectively.

\begin{figure*}[h!]
\centering
\includegraphics[width=0.9\textwidth, trim=0.1cm 0cm 0.1cm 0cm, clip=true]{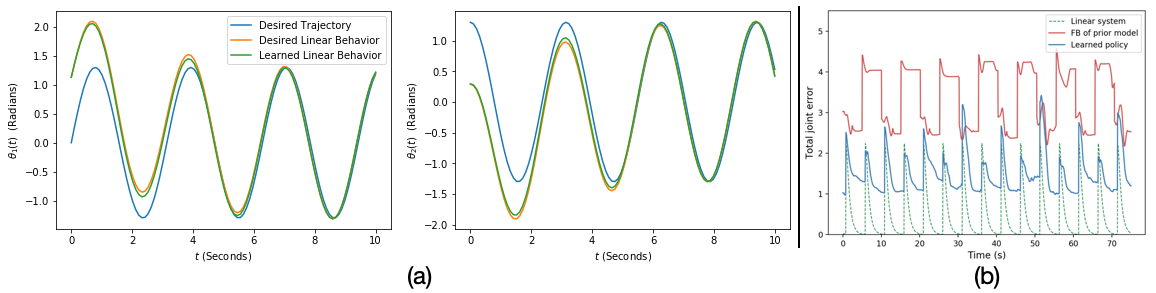}
\caption{(a) Depiction of the tracking task for the double pendulum. The blue curves represent the desired joint angles over time. The orange and green curves represent trajectories generated by the ideal feedback linearizing controller (perfect model information) and our learned policy. Both trajectories begin with an initial tracking error, but quickly converge to the desired motion.  (b) Total $\ell_2$ error on a set-point tracking task for the Baxter robot after 104 minutes of training for the desired linear system is (dotted, green), the nominal feedback linearizing controller (red), and our learned controller (blue)}
\label{fig:double_pendulum}
\vspace{-0.6cm}
\end{figure*}

Altogether, for a given $\theta =(\theta_1,\theta_2) \in \Theta$ our estimate for the controller which exactly linearizes the plant is given by
\begin{equation}\label{eq:learned_controller}
    \hat{u}_\theta(x,v) = [\beta_m(x) + \beta_{\theta_1}(x)] + [\alpha_{m}(x) + \alpha_{\theta_2}(x)  ] v
\end{equation}
When no prior information about the dynamics of the plant is available (other than its vector relative degree), we simply set $\beta_m \equiv 0$ and $\alpha_m \equiv 0$ in the above expression. Next we define an optimization problem which selects the parameters for the learned controller which best linearize the plant.

\subsection{Continuous-time optimization problem}\label{sec:continuous_opt}
From Section \ref{sec:FBL} we know that the input-output dynamics of the plant are of the form
\begin{equation}\label{eq:learned_output_dyn}
   y^{(\gamma)} = b_p(x) +  A_p(x)u
\end{equation}
where the terms $b_p$ and $A_p$ are unknown to us, and we have written the highest order derivatives of the outputs as $y^{(\gamma)} = (y_1^{(\gamma_1)} \dots ,y_q^{(\gamma_q)})^T$ to simplify notation. Under application of $\hat{u}_\theta$ the dynamics are given by:
\begin{equation}\label{eq:total_dynamics}
    y^{(\gamma)} = \underbrace{b_p(x) +  A_p(x) \hat{u}_\theta(x,v)}_{W_\theta(x,v) }
\end{equation}
Since our goal is to find parameters which linearize the plant, we want to find $\theta^* \in \Theta$ such that $W_{\theta^*}(x,v) \approx v$  for each  $x \in D$ and $v \in \R^q$. Thus, we define the point-wise loss $\ell \colon \R^n \times \R^q \times \R^{K_1 + K_2}\to \R$ by
\begin{equation}
    \ell(x,v, \theta) = \norm{v -W_\theta(x,v)}_2^2,
\end{equation}
which provides a measure of how well the learned controller $\hat{u}_\theta$ linearizes the plant at the state $x$ when the virtual input $v$ is applied to the linear reference model. Next, we specify a probability distribution $X$ over $\R^n$ with support $D \subset \mathbb{R}^n$ and let $V$ be the uniform distribution over the set $\set{v \in \R^q \colon \norm{v} \leq 1}$. We then define the weighted loss
\begin{equation}\label{eq:pointwise_loss}
    L(\theta) = \expval_{x \sim X, v \sim V} \ell(x,v,\theta)
\end{equation}
and select our optimal choice for the parameters of the learned controller via the following optimization:
 \begin{equation}\label{eq:continuous_opt}
   (\textbf{P}) \colon \  \min_{\theta \in \Theta} L(\theta)
\end{equation}

Here, the distribution $X$ models our preference for having an accurate linearizing controller at different points in the state-space, and the uniformity of $V$ ensures that optimal solutions of \textbf{P} accurately linearize the plant for all possible choices of the virtual input to the reference model. The primary challenge in solving \textbf{P} is that we do not know the terms in $W_\theta(x,v)$, since we do not know $A_p(x)$ and $b_p(x)$. However, using the relationship \eqref{eq:total_dynamics}, we can query $W_{\theta}(x,v)$ by measuring $y^{(\gamma)}$ when different inputs are applied to the plant. Thus, \textbf{P} can be solved by running experiments on the plant and using any stochastic optimization method which only requires access to the point-wise loss \eqref{eq:pointwise_loss}. While we focus on the use of policy-gradient reinforcement learning algorithms below, there are many possible model-free approaches for solving \textbf{P} which merit further investigation. Importantly, since we have abstracted away the need for a model when solving $\textbf{P}$, we can iteratively improve the performance of our learned controller without requiring that it remain invertible at each stage of the optimization. Next, we discuss when we can recover the true linearizing controller for the plant by solving \textbf{P}:

\begin{lemma}\vspace{0.1cm}
Suppose that there exists $\theta^* \in \Theta$ such that $\hat{u}_{\theta^*}(x,v)  = u_p(x,v)$ for each $x \in D$ and $v \in V$. Then $\theta^*$ is a globally optimal solution of \textbf{P}.
\end{lemma}
\begin{proof}
Note that if $u_{\theta^*}(x,v) = u_p(x,v)$ for each $x \in X$ and $v \in V$  then $L(\theta) = 0$. Moreover, we clearly have $L(\theta) \geq 0$ for each $\theta \in \Theta$. Thus, $\theta^*$ must be a global minimizer of the optimization \eqref{eq:continuous_opt}. \vspace{0.1cm}
\end{proof}

However, \textbf{P} is generally non-convex meaning that in practice we can only hope to find locally optimal solutions to the problem. Thus, we seek conditions which simplify the structure of the optimization. The standard convergence proofs in the adaptive control literature assume the learned controller is of the form
\begin{equation}\label{eq:lin_param1}
    \beta_{\theta_1}(x) = \sum_{k = 1}^{K_1} \theta_k^1 \beta_{k}(x) \ \  \ \  \ \    \alpha_{\theta_2}(x) = \sum_{k =1}^{K_2} \theta_k^2\alpha_{k}(x)
\end{equation}
where $\set{\beta_k}_{k =1}^{K_1}$
and $\set{\alpha_k}_{k=1}^{K_2}$ are nonlinear continuous functions. When we adopt this structure our optimization becomes convex, whose proof can be found in the Appendix: 

\begin{lemma}\label{lemma:convex1}\vspace{0.1cm}
Assume that $\beta_{\theta_1}$ and $\alpha_{\theta_2}$ are of the form \eqref{eq:lin_param1}. Then \textbf{P} is convex. Moreover, if $\set{\beta_k}_{k=1}^{K_1}$ and $\set{\alpha_k}_{k=1}^{K_2}$ are each linearly independent  then \textbf{P} is strongly convex. 
\vspace{0.1cm}
\end{lemma}

Thus, when our learned controller takes the form \eqref{eq:lin_param1} we can reliably use iterative techniques to find its globally optimal solution. The proof of Lemma \ref{lemma:convex1}, which is given in the accompanying technical report, shows that $L$ is actually quadratic in the parameters when the learned controller is of the form \eqref{eq:lin_param1}. The key property being exploited here is that $W_{\theta}(x,v)$ becomes affines in $\theta$ when the controller is linear in the parameters. Our conditions on $X$ and $V$ are analogous to the persistency of excitation results from the adaptive control literature \cite[Chapter 2]{sastry1989adaptive}, and ensure that each component of the learned controller is excited while solving \textbf{P}. This is the underlying reason why the optimization becomes strongly convex when the components of our controller are linearly independent. Taken together, the preceding lemmas imply our main theoretical result: 
\begin{theorem}\vspace{0.1cm}
Suppose that for some $\theta^* \in \Theta$ we have $\hat{u}_{\theta^*}(x,v) = u_p(x,v)$ for each $x \in D$ and $v \in \R^q$. Further assume that the learned controller is of the form \eqref{eq:lin_param1} and that the sets $\set{\beta_k}_{k =1}^{K_1}$
and $\set{\alpha_k}_{k=1}^{K_2}$ are linearly independent. Then $\theta^*$ is the unique global (and local) minimizer of \textbf{P}.
\label{thm:theorem1} \vspace{0.1cm}
\end{theorem}

There are a number of well-studied bases, such as polynomials or radial basis functions, which can be used to approximate any continuous function to a desired degree of accuracy. Thus, by including enough terms in  $\set{\beta_k}_{k =1}^{K_1}$
and $\set{\alpha_k}_{k=1}^{K_2}$ we can theoretically recover $u_p$ to a pre-specified precision by solving \textbf{P}. However, for high dimensional systems, the number of terms required in such expansions can become prohibitively large. Other architectures, such as multi-layer feed-forward neural networks, yield more compactly represented function approximation schemes but complicate the analysis of \textbf{P}. However, since our approach does not suffer from singularities during the optimization process, in practice we can still use these powerful function approximators to find an improved linearizing controller for high-dimensional systems with unknown dynamics. We next demonstrate this point by discussing how policy optimization algorithms can be used to solve a discrete-time approximation to \textbf{P}.

\begin{figure*}[h!]
\centering
\includegraphics[width=0.95\textwidth]{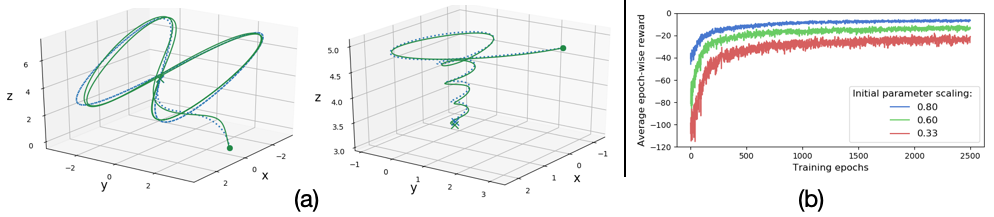}
\caption{(a) The performance of our learned linearizing controller on two high-performance quadrotor tracking tasks when the parameters of the inital dynamics model are scaled by a factor of $0.6$.  The first task is a figure-eight, and the second is a corkscrew maneuver. In both maneuvers the quadrotor also tracks an oscillating reference in the yaw angle. In both figures, the green trajectory is the one taken by the learned controller and the blue trajectory is that of the true linearizing controller for the system. In both cases, the learned controller closely matches the desired behavior. (b) Learning curves for learned linearizing controllers which are instantiated with nominal dynamics models with three levels of scaling of the true dynamics parameters. }
\label{fig:quad}
\vspace{-0.6cm}
\end{figure*}

\subsection{Discrete-time approximations and reinforcment learning}\label{sec:RL_problem}
While the theory and optimization problem we developed in the previous section are in continuous time, many real world plants have actuators which can only be updated at a fixed sampling frequency. Thus in this section we formulate a discrete time approximation to \textbf{P} which we cast as a canonical reinforcement learning problem. Unfortunately, the analysis of linearizable systems in the sampled-data setting becomes significantly more complex \cite{grizzle1988feedback}, which is why the majority of the literature on adaptive control for linearizable systems is formulated in continuous time. 

Letting $\Delta t>0$ denote the sampling rate, we set $t_k = k\Delta t$ for each $k \in \N$. Next, letting $x(\cdot)$ denote the state trajectory of the plant, we then denote $x_k = x(t_k)$ and let $u_k$ denote the control applied on the interval $[t_k, t_{k+1})$. This yields the following difference equation for the dynamics of the plant:
\begin{equation*}
x_{k+1} =x_{k} +  \underbrace{\int_{t_k}^{t_{k+1}}f_p(x(t)) + g_p(x(t))u_k d\tau}_{F(x_k,u_k)}
\end{equation*}
The map $F_p \colon \R^n \times \mathbb{R}^q \to \mathbb{R}^n$ will generally no longer be affine in the control. Similarly, we let $\xi(\cdot) = (y_1(\cdot), \dots, y_1^{(\gamma_1 -1)}(\cdot), \dots,y_q(\cdot), \dots y_q^{(\gamma_q -1)}(\cdot))$ denote the solution to the linearized portion of the state, and set $\xi_k = \xi(t_k)$ for each iterate. Now, if we integrate \eqref{eq:learned_output_dyn} over the interval $[t_k, t_{k+1})$ we obtain a difference equation for the outputs of the form
\begin{equation} \label{eq:dt_io}
\xi_{k+1} = e^{A\Delta t}\xi_k + H(x_k, u_k)
\end{equation}
where again the mapping $H \colon \mathbb{R}^n \times \R^q \to \R^{\gamma}$ will generally no longer be affine in $u_k$. This is the primary hurdle to applying the theory developed in the previous section to the current setting. Nevertheless, we can still attempt to enforce an approximate linear relationship between a virtual input $v_k$ and successive iterates of the outputs. Letting  $\bar{A} = e^{A\Delta t}$ and $\bar{B} = \int_{0}^{\Delta t}e^{At}B dt$, the sampled-data version of the reference model \eqref{eq:reference} becomes
\begin{equation*}
    \xi_{k+1} = \bar{A}\xi_k + \bar{B}v_k,
\end{equation*}
which is the ideal linear behavior we would like to enforce by applying the control $u_k = \hat{u}_\theta(x_k,v_k)$.  To encourage actions which better track the discrete-time reference model, we apply the point-wise loss $\bar{\ell} \colon \R^n \times \R^q \times \R^q \to \R$ where
\begin{equation*}
    \bar{\ell}(x_k, u_k, v_k) = \| \bar{B}v_k - H(x_k, u_k) \|_2^2.
\end{equation*}
We can calculate this quantity using \eqref{eq:dt_io} and by observing $\xi_k$ and $\xi_{k+1}$, which can each be calculated by numerically differentiating the outputs from the plant. This loss provides a measure of how well the control $u_k$ enforces the desired change in the state of the linear system (as specified by $v_k$) at the state $x_k$. We then define the following reinforcement learning problem over the parameters of $\hat{u}_{\theta}$:
\begin{align*}
\min_{\theta \in \Theta} \mathbb{E}_{x_0 \sim X, v_k \sim V, w_k \sim \mathcal{N}(0, \sigma_w^2)} &\left[\sum_{k=1}^N \bar{\ell}(x_k, u_k, v_k)\right] \\
 \text{subject to: }  \ \ \ x_{k+1} &= x_k +  F(x_k, u_k), \ \ \ x_0 = x_0 \\ 
 u_k &= \hat{u}_\theta(x_k, v_k) + w_k
\end{align*}
Here, we sample initial conditions for the problem using our desired state distribution $X$ and sample inputs to the linear system according to $V$ at each time step. The zero mean added noise $w_k$ is used to encourage exploration, and to make the effects of the policy random. Finally, $N$ is the length of the training episodes. This problem can be solved to local optimality using standard reinforcement learning algorithms~\cite{SuttonBarto,TRPO,PPO,DDPO} and by running experiments on the real-world hardware to evaluate $\bar{l}(x_k,u_k,v_k)$. While it is not immediately clear how to extend the theoretical results from Section \ref{sec:continuous_opt} to the present case, we intend to address this issue in a forthcoming article.


\section{Examples}\label{sec:examples}

We now use our approach to learn feedback linearizing policies for three different systems and use the learned policies to construct tracking controllers as in \cite[Theorem 9.14]{sastry2013nonlinear}. In all cases, the input to the parameterized policy replaces all angles with their sine and cosine. 
\subsection{Simulations}

\subsubsection{Double pendulum with polynomial policies}
\label{sec:dp}
We first test our approach on a fully actuated double pendulum with state $x = [\theta_1, \theta_2, \omega_1, \omega_2]^T$, output $y = [\theta_1, \theta_2]^T$, where
$\theta_1$ and $\theta_2$ represent the angles of the two joints, with angular rates $\omega_1$ and $\omega_2$ respectively. The system has two inputs $u_1$ and $u_2$ that control the torque at both joints. Although the system is relatively low dimensional and fully actuated, it is highly nonlinear \cite{shinbrot1992chaos}. 

For the learning problem, we give the nominal model inaccurate estimates for both the length and mass of each of the arms. Specifically, we scale each of the parameters to $1/2$ their true values. The learned controller is comprised of 150 radial basis functions, which are centered randomly throughout the state-space. The learned controller is linear in its parameters, so as to agree with Theorem~\ref{thm:theorem1}. We use the REINFORCE algorithm \cite{SuttonBarto}, and baseline state-value estimates with the average reward over all states. At each iteration (or epoch) we collect $50$ rollouts of $0.25$ seconds each, and we train for 500 epochs. Figure~\ref{fig:double_pendulum} (a) presents the result of using our learned controller to track a desired sinusoidal trajectory, and compares it to the performance of the exact linearizing controller for the system. The difference between the performance our learned controller and the theoretical ideal observed to be very small. We do not plot the trajectory generated by the nominal model-based controller, since it immediately diverges from the desired behavior. Similar results were observed using the same learned controller to track other desired trajectories for the system. The linear feedback matrix used in the tracking controller was obtained by solving an infinite horizon LQR problem where the state deviation was penalized 10 times more than the magnitude of the input.

\subsubsection{14D quadrotor with neural network policies}
Our second simulation environment uses the quadrotor model and feedback linearization controller proposed in \cite{al2009quadrotor}, which makes use of dynamic extension \cite{sastry2013nonlinear}. In particular, the states for the model are $(x,y,z,\psi,\theta,\phi, \dot{x}, \dot{y}, \dot{z}, p,q,r, \xi,\zeta)$ where $x$, $y$ and $z$ are the Cartesian coordinates of the quadrotor, and $\psi$, $\theta$ and $\phi$ represent the roll, pitch and yaw of the quadrotor, respectively. The next six states represent the time derivatives of these state: $\frac{d}{dt}(x,y,z,\psi,\theta,\phi) = (\dot{x}, \dot{y}, \dot{z}, p,q,r)$. Finally, $\xi$ and $\zeta$ are the extra states obtained from the dynamic extension procedure. The outputs for the model are the $x$, $y$, $z$ and $\psi$ coordinates.

The inaccurate nominal dynamics model was constructed by multiplying the mass and inertial constants of the true system by factors of $0.33$, $0.6$ and $0.8$ for different experiments. The trained policies were feed-forward neural networks with $\tanh$ activations with $2$ hidden layers of width $64$. For each training epoch, 50 rollouts of length 25 were collected and the parameters were updated using PPO. We trained both policies for 2500 epochs. Figure~\ref{fig:quad} (b) illustrates how a better initial dynamics model leads to faster learning of an accurate linearizing controller. In particular, the reward-per-epoch is plotted for policies trained with the three scaling factors indicated above. We observe that worse initial models result in worse policy performance, given the same network architecture and training time.  Figure~\ref{fig:quad} (a) demonstrates the ability of the learned controller to overcome significant model mismatch (scaling factor of $0.6$) to match the desired linear tracking behavior.

\subsection{Robotic experiment: 7-DOF manipulator arm}
\label{subsec:baxter}

We also evaluate our approach in hardware, on a 7-DOF Baxter robot arm.
The dynamics of this 14-dimensional system are extremely coupled and nonlinear.
Taking the 7 joint angles as output $y$, however, the system is input-output linearizable with relative degree two.
We use the system measurements (i.e., masses, link lengths, etc.) provided with Baxter's pre-calibrated URDF \cite{robotics2013baxter} and the OROCOS Kinematics and Dynamics Library (KDL) \cite{bruyninckx2001open} to compute a nominal feedback linearizing control law.

This nominal controller suffers from several inaccuracies.
First, Baxter's actuators are series-elastic, meaning that each joint contains a torsion spring \cite{williams2017baxter} which is unmodeled, and the URDF itself may not be perfectly accurate.
Second, the OROCOS solver is numerical, which can lead to errors in computing the decoupling matrix and drift term.
Finally, our control architecture is implemented in the Robot Operating System \cite{quigley2009ros}, which can lead to minor timing inconsistency.

We use the PPO algorithm to tune the parameters of a $128 \times 2$ neural network with $\tanh$ activations for the learned component of the controller. For each training epoch, 1250 rollouts of one timestep (0.05 s) each were collected. We trained for 100 epochs, which took 104 minutes.
Figure~\ref{fig:double_pendulum} (b) summarizes typical results on tracking a square wave reference trajectory for each joint angle with period 5s. The nominal feedback linearized model from OROCOS has significant steady-state error. Our learned approach significantly reduces, but does not eliminate, this error. 
We conjecture the remaining error is a sign that the (relatively small) neural network may not be sufficiently expressive. 

\section{Discussion}\label{sec:discussion}
While the methods we have presented avoid some of the challenges model-based methods face when trying to construct a linearizing controller for an unknown plant, we feel the techniques presented here should be used to complement model-based design approaches rather than replace them. As we observed empirically, a reasonable (though ultimately inaccurate) model can be used to provide a better starting point for the learning process. Thus, we feel that in practice our approach should be used primarily to overcome difficult to model non-linearities, and should be used in combination with a simple nominal dynamics models with parameters which are easily identified. Investigating the performance of different learning algorithms in real-world settings merits significant future research effort. On the theoretical side, we feel the general way in which we constructed the optimization problem in \ref{sec:continuous_opt} provides a foundation for combining model-free policy optimization techniques with geometric control architectures. Future work will investigate how to optimize over different classes of controllers by extending the approach presented here.

\appendix
\subsection{Proof of Lemma \ref{lemma:convex1}}
First, we rearrange \eqref{eq:learned_output_dyn} into the form
\begin{multline}
    y_p^\gamma =  b_p(x) + A_p(x)\beta_m(x) + A_p(x) \alpha_m(x)v  \\ + \sum_{i =1}^{K_1}\theta_k^1 A_p(x)\beta_k(x) + \sum_{k =1}^{K_2}\theta_k^2 A_p(x)\alpha_k(x)v
\end{multline}
to separate out the portions that depend on $\theta$. This can be further condensed by putting $y^\gamma = \bar{W}(x,v) + \hat{W}(x,v)\theta$
where we set
\begin{equation*}
\bar{W}(x,v) = b_p(x) +A_p(x)\beta_m(x) + A_p(x)\alpha_m(x)v
\end{equation*}
\begin{multline*}
\hat{W}(x,v)=A_p(x) [\beta_1(x), \dots, \beta_{K_1}(x), A_1(x)v, \\  \dots, A_{K_2}(x)v ]
\end{multline*}
Letting $c(x,v) = (v- \bar{W}(x,v))$, we can rewrite
\begin{align*}
    \ell(x,v,\theta) &= \big( c(x,v) - \hat{W}(x,v) \theta \big)^T\big(c(x,v) -  \hat{W}(x,v) \theta \big)\\
    &= \theta^T\hat{W}^T(x,v)\hat{W}(x,v)\theta  - 2\theta \hat{W}(x,v) c(x,v) \\ & \ \ \ \ \ \  \ \ \ \ \ \ \ \ \ \ \ \ \ \ \ +c(x,v)^T c(x,v)
\end{align*}
From here we observe that $L(\theta) = \theta^T W \theta + \theta^T F + d$ where $W = \expval_{x \sim X, v \sim V} \hat{W}(x,v)^T \hat{W}(x,v)$ is a positive semi-definite matrix, $F = \expval_{x \sim X, v \sim V}\hat{W}(x,v) c(x,v)$ and $d = \expval_{x \sim X, v \sim V} c(x,v)^Tc(x,v)$. Thus, recalling that $\Theta$ is assumed to be a convex set, we see that \eqref{eq:continuous_opt} is a convex optimization problem which will be strictly convex  if, and only if, $W$ is positive definite. Letting $w^1_k\colon (x,v) \to A_p(x)\beta_k(x)$ and $w^2_{k} \colon (x,v) \to A_p(x)\alpha_k(x)v$,
we see that $W$ is nothing but the Grammian of the set $\Omega = \set{w_1^1, \dots, w_{K_1}^1,\omega^2_1, \dots, w_{K_2}^2}$ on $C(D\times B_v, \R^q)$ with respect to an inner product which is weighted by the distributions $X$ and $V$. Thus, $W$ will be positive definite if and only if $\Omega$ is linearly independent on $C(D \times B_v, \R^q)$. For the sake of contradiction assume that $\Omega$ is not linearly independent. Then there exists scalars $\set{c_k^1}_{k=1}^2$ and $\set{c_{k}^2}_{k=1}^{K_2}$ such that for each $x \in D$ and $v \in V$
 \begin{equation}
     \sum_{k =1}^{K_1} c_k^1\omega_k^1(x,v) + \sum_{k=1}^{K_2}c_k^2\omega_k^2(x,v) = 0
 \end{equation}
Since we know that $A_p(x)$ is invertible for each $x \in D$, this statement is equivalent to 
\begin{equation}
   \sum_{k =1}^{K_1} c_k^1\beta_k(x) + \sum_{k=1}^{K_2}c_k^2\alpha_k(x)v = 0.
\end{equation}
holding for each $x \in D$ and $v \in V$. However, it is not difficult to see that this condition is ruled out in the case that $\set{\beta_k}_{k =1}^{K_1}$ and $\set{\alpha_k}_{k=1}^{K_2}$ are linearly independent sets.

\bibliographystyle{IEEEtran}
\bibliography{refs.bib}

\begin{thebibliography}{10}
\providecommand{\url}[1]{#1}
\csname url@samestyle\endcsname
\providecommand{\newblock}{\relax}
\providecommand{\bibinfo}[2]{#2}
\providecommand{\BIBentrySTDinterwordspacing}{\spaceskip=0pt\relax}
\providecommand{\BIBentryALTinterwordstretchfactor}{4}
\providecommand{\BIBentryALTinterwordspacing}{\spaceskip=\fontdimen2\font plus
\BIBentryALTinterwordstretchfactor\fontdimen3\font minus
  \fontdimen4\font\relax}
\providecommand{\BIBforeignlanguage}[2]{{%
\expandafter\ifx\csname l@#1\endcsname\relax
\typeout{** WARNING: IEEEtran.bst: No hyphenation pattern has been}%
\typeout{** loaded for the language `#1'. Using the pattern for}%
\typeout{** the default language instead.}%
\else
\language=\csname l@#1\endcsname
\fi
#2}}
\providecommand{\BIBdecl}{\relax}
\BIBdecl

\bibitem{sastry2013nonlinear}
S.~Sastry, \emph{Nonlinear systems: analysis, stability, and control}.\hskip
  1em plus 0.5em minus 0.4em\relax Springer Science \& Business Media, 1999,
  vol.~10.

\bibitem{isidori2013nonlinear}
A.~Isidori, \emph{Nonlinear control systems}.\hskip 1em plus 0.5em minus
  0.4em\relax Springer Science \& Business Media, 2013.

\bibitem{bertsekas1996neuro}
D.~P. Bertsekas and J.~N. Tsitsiklis, \emph{Neuro-dynamic programming}.\hskip
  1em plus 0.5em minus 0.4em\relax Athena Scientific, 1996.

\bibitem{sutton2018reinforcement}
R.~S. Sutton and A.~G. Barto, \emph{Reinforcement learning: An introduction},
  2018.

\bibitem{sutton2000policy}
R.~S. Sutton, D.~A. McAllester, S.~P. Singh, and Y.~Mansour, ``Policy gradient
  methods for reinforcement learning with function approximation,'' in
  \emph{Advances in neural information processing systems}, 2000, pp.
  1057--1063.

\bibitem{schulman2015trust}
J.~Schulman, S.~Levine, P.~Abbeel, M.~Jordan, and P.~Moritz, ``Trust region
  policy optimization,'' in \emph{International conference on machine
  learning}, 2015, pp. 1889--1897.

\bibitem{lillicrap2015continuous}
T.~P. Lillicrap, J.~J. Hunt, A.~Pritzel, N.~Heess, T.~Erez, Y.~Tassa,
  D.~Silver, and D.~Wierstra, ``Continuous control with deep reinforcement
  learning,'' \emph{arXiv preprint arXiv:1509.02971}, 2015.

\bibitem{schulman2017proximal}
J.~Schulman, F.~Wolski, P.~Dhariwal, A.~Radford, and O.~Klimov, ``Proximal
  policy optimization algorithms,'' \emph{arXiv preprint arXiv:1707.06347},
  2017.

\bibitem{martin2003flat}
P.~Martin, R.~M. Murray, and P.~Rouchon, ``Flat systems, equivalence and
  trajectory generation,'' 2003.

\bibitem{kalman1960contributions}
R.~E. Kalman \emph{et~al.}, ``Contributions to the theory of optimal control,''
  \emph{Bol. soc. mat. mexicana}, vol.~5, no.~2, pp. 102--119, 1960.

\bibitem{borrelli2017predictive}
F.~Borrelli, A.~Bemporad, and M.~Morari, \emph{Predictive control for linear
  and hybrid systems}.\hskip 1em plus 0.5em minus 0.4em\relax Cambridge
  University Press, 2017.

\bibitem{grizzle2001asymptotically}
J.~W. Grizzle, G.~Abba, and F.~Plestan, ``Asymptotically stable walking for
  biped robots: Analysis via systems with impulse effects,'' \emph{IEEE
  Transactions on automatic control}, vol.~46, no.~1, pp. 51--64, 2001.

\bibitem{ames2014rapidly}
A.~D. Ames, K.~Galloway, K.~Sreenath, and J.~W. Grizzle, ``Rapidly
  exponentially stabilizing control lyapunov functions and hybrid zero
  dynamics,'' \emph{IEEE Transactions on Automatic Control}, vol.~59, no.~4,
  pp. 876--891, 2014.

\bibitem{mellinger2011minimum}
D.~Mellinger and V.~Kumar, ``Minimum snap trajectory generation and control for
  quadrotors,'' in \emph{2011 IEEE International Conference on Robotics and
  Automation}.\hskip 1em plus 0.5em minus 0.4em\relax IEEE, 2011, pp.
  2520--2525.

\bibitem{sastry2011adaptive}
S.~Sastry and M.~Bodson, \emph{Adaptive control: stability, convergence and
  robustness}.\hskip 1em plus 0.5em minus 0.4em\relax Courier Corporation,
  1989.

\bibitem{craig1987adaptive}
J.~J. Craig, P.~Hsu, and S.~S. Sastry, ``Adaptive control of mechanical
  manipulators,'' \emph{The International Journal of Robotics Research},
  vol.~6, no.~2, pp. 16--28, 1987.

\bibitem{sastry1989adaptive}
S.~S. Sastry and A.~Isidori, ``Adaptive control of linearizable systems,''
  \emph{IEEE Transactions on Automatic Control}, vol.~34, no.~11, pp.
  1123--1131, 1989.

\bibitem{nam1988model}
K.~Nam and A.~Araposthathis, ``A model reference adaptive control scheme for
  pure-feedback nonlinear systems,'' \emph{IEEE Transactions on Automatic
  Control}, vol.~33, no.~9, pp. 803--811, 1988.

\bibitem{kanellakopoulos1991systematic}
I.~Kanellakopoulos, P.~V. Kokotovic, and A.~S. Morse, ``Systematic design of
  adaptive controllers for feedback linearizable systems,'' in \emph{1991
  American Control Conference}.\hskip 1em plus 0.5em minus 0.4em\relax IEEE,
  1991, pp. 649--654.

\bibitem{umlauft2017feedback}
J.~Umlauft, T.~Beckers, M.~Kimmel, and S.~Hirche, ``Feedback linearization
  using gaussian processes,'' in \emph{2017 IEEE 56th Annual Conference on
  Decision and Control (CDC)}.\hskip 1em plus 0.5em minus 0.4em\relax IEEE,
  2017, pp. 5249--5255.

\bibitem{chowdhary2014bayesian}
G.~Chowdhary, H.~A. Kingravi, J.~P. How, and P.~A. Vela, ``Bayesian
  nonparametric adaptive control using gaussian processes,'' \emph{IEEE
  Transactions on Neural Networks and Learning Systems}, vol.~26, no.~3, pp.
  537--550, 2014.

\bibitem{chowdhary2013bayesian}
------, ``Bayesian nonparametric adaptive control of time-varying systems using
  gaussian processes,'' in \emph{2013 American Control Conference}.\hskip 1em
  plus 0.5em minus 0.4em\relax IEEE, 2013, pp. 2655--2661.

\bibitem{spooner1996stable}
J.~T. Spooner and K.~M. Passino, ``Stable adaptive control using fuzzy systems
  and neural networks,'' \emph{IEEE Transactions on Fuzzy Systems}, vol.~4,
  no.~3, pp. 339--359, 1996.

\bibitem{chen1995adaptive}
F.-C. Chen and H.~K. Khalil, ``Adaptive control of a class of nonlinear
  discrete-time systems using neural networks,'' \emph{IEEE Transactions on
  Automatic Control}, vol.~40, no.~5, pp. 791--801, 1995.

\bibitem{yesildirek1994feedback}
A.~Yesildirek and F.~L. Lewis, ``Feedback linearization using neural
  networks,'' in \emph{Proceedings of 1994 IEEE International Conference on
  Neural Networks (ICNN'94)}, vol.~4.\hskip 1em plus 0.5em minus 0.4em\relax
  IEEE, 1994, pp. 2539--2544.

\bibitem{kosmatopoulos1999switching}
E.~B. Kosmatopoulos and P.~A. Ioannou, ``A switching adaptive controller for
  feedback linearizable systems,'' \emph{IEEE Transactions on automatic
  control}, vol.~44, no.~4, pp. 742--750, 1999.

\bibitem{kosmatopoulos2002robust}
------, ``Robust switching adaptive control of multi-input nonlinear systems,''
  \emph{IEEE transactions on automatic control}, vol.~47, no.~4, pp. 610--624,
  2002.

\bibitem{bechlioulis2008robust}
C.~P. Bechlioulis and G.~A. Rovithakis, ``Robust adaptive control of feedback
  linearizable mimo nonlinear systems with prescribed performance,'' \emph{IEEE
  Transactions on Automatic Control}, vol.~53, no.~9, pp. 2090--2099, 2008.

\bibitem{umlauft2019feedback}
J.~Umlauft and S.~Hirche, ``Feedback linearization based on gaussian processes
  with event-triggered online learning,'' \emph{IEEE Transactions on Automatic
  Control}, 2019.

\bibitem{grizzle1988feedback}
J.~Grizzle and P.~Kokotovic, ``Feedback linearization of sampled-data
  systems,'' \emph{IEEE Transactions on Automatic Control}, vol.~33, no.~9, pp.
  857--859, 1988.

\bibitem{SuttonBarto}
R.~S. Sutton and A.~G. Barto, \emph{Introduction to Reinforcement Learning},
  1st~ed.\hskip 1em plus 0.5em minus 0.4em\relax Cambridge, MA, USA: MIT Press,
  1998.

\bibitem{TRPO}
J.~Schulman, S.~Levine, P.~Abbeel, M.~Jordan, and P.~Moritz, ``Trust region
  policy optimization,'' in \emph{International Conference on Machine
  Learning}, 2015, pp. 1889--1897.

\bibitem{PPO}
J.~Schulman, F.~Wolski, P.~Dhariwal, A.~Radford, and O.~Klimov, ``Proximal
  policy optimization algorithms,'' \emph{CoRR}.

\bibitem{DDPO}
D.~Silver, G.~Lever, N.~Heess, T.~Degris, D.~Wierstra, and M.~Riedmiller,
  ``Deterministic policy gradient algorithms,'' in \emph{Proceedings of the
  31st International Conference on Machine Learning}, ser. Proceedings of
  Machine Learning Research, 2014, pp. 387--395.

\bibitem{shinbrot1992chaos}
T.~Shinbrot, C.~Grebogi, J.~Wisdom, and J.~A. Yorke, ``Chaos in a double
  pendulum,'' \emph{American Journal of Physics}, vol.~60, no.~6, pp. 491--499,
  1992.

\bibitem{al2009quadrotor}
S.~A. Al-Hiddabi, ``Quadrotor control using feedback linearization with dynamic
  extension,'' in \emph{2009 6th International Symposium on Mechatronics and
  its Applications}.\hskip 1em plus 0.5em minus 0.4em\relax IEEE, 2009, pp.
  1--3.

\bibitem{robotics2013baxter}
R.~Robotics, ``Baxter,'' \emph{Retrieved Jan}, vol.~10, p. 2014, 2013.

\bibitem{bruyninckx2001open}
H.~Bruyninckx, ``Open robot control software: the orocos project,'' in
  \emph{Proceedings 2001 ICRA. IEEE international conference on robotics and
  automation (Cat. No. 01CH37164)}, vol.~3.\hskip 1em plus 0.5em minus
  0.4em\relax IEEE, 2001, pp. 2523--2528.

\bibitem{williams2017baxter}
R.~L. Williams~II, ``Baxter humanoid robot kinematics{\copyright} 2017 dr. bob
  productions robert l. williams ii, ph. d., williar4@ ohio. edu mechanical
  engineering, ohio university, april 2017,'' 2017.

\bibitem{quigley2009ros}
M.~Quigley, K.~Conley, B.~P. Gerkey, J.~Faust, T.~Foote, J.~Leibs, R.~Wheeler,
  and A.~Y. Ng, ``Ros: an open-source robot operating system,'' in \emph{ICRA
  Workshop on Open Source Software}, 2009.

\end{thebibliography}

\end{document}